\documentclass[a4paper,12pt,reqno]{amsart}
\usepackage[T1]{fontenc}
\usepackage[utf8]{inputenc}
\usepackage{bera}
\usepackage{amssymb,dsfont,mathrsfs}
\usepackage[margin=1in]{geometry}
\usepackage{enumitem}
\usepackage{graphicx,color}
\usepackage{xparse}
\usepackage[bookmarksdepth=2]{hyperref}

\newtheorem*{theorem}{Theorem}

\newcommand{\R}{\mathds{R}}

\newcommand{\ph}{\varphi}

\NewDocumentCommand{\formula}{ssom}{%
 \IfBooleanTF{#1}{%
  \IfBooleanTF{#2}{%
   \IfValueTF{#3}%
    {\begin{align}\label{#3}\begin{gathered}#4\end{gathered}\end{align}}%
    {\begin{gather}#4\end{gather}}%
  }{%
   \IfValueTF{#3}%
    {\begin{align}\label{#3}\begin{aligned}#4\end{aligned}\end{align}}%
    {\begin{gather*}#4\end{gather*}}%
  }%
 }{%
  \IfValueTF{#3}%
   {\begin{align}\label{#3}#4\end{align}}%
   {\begin{align*}#4\end{align*}}%
 }%
}

\begin{document}

\title{Simplicity of eigenvalues of the fractional \\ Laplace operator in an interval}
\author{Mateusz Kwaśnicki}
\thanks{Work supported by the Polish National Science Centre (NCN) grant no.\@ 2019/33/B/ST1/03098}
\address{Mateusz Kwaśnicki \\ Department of Pure Mathematics \\ Wrocław University of Science and Technology \\ ul. Wybrzeże Wyspiańskiego 27 \\ 50-370 Wrocław, Poland}
\email{mateusz.kwasnicki@pwr.edu.pl}

\begin{abstract}
We give a short proof of simplicity of the eigenvalues of the fractional Laplace operator in an interval, a result shown recently by Fall, Ghimenti, Micheletti and Pistoia [Calc.\@ Var.\@ Partial Differ.\@ Equ.\@ 62 (2023), \#233].
\end{abstract}

\maketitle
\thispagestyle{empty}

%
%

The following theorem is one of the main results of~\cite{fgmp}. The original proof is rather involved. We propose a very short, direct argument. Noteworthy, the only tool needed here, the fractional Pokhozhaev identity, has been readily available for a decade.

\begin{theorem}[Corollary~2 in~\cite{fgmp}]
If $s \in (0, 1)$, the eigenvalues of the spectral problem
\formula[eq:pro]{
 \tag{$\diamondsuit$}
 \begin{cases}
  (-\Delta)^s \ph = \lambda \ph & \text{in $(0, 1)$,} \\
  \ph = 0 & \text{in $\R \setminus (0, 1)$,}
 \end{cases}
}
are simple.
\end{theorem}

\begin{proof}
By the fractional Pokhozhaev identity (Theorem~1.1 of~\cite{rs}), \eqref{eq:pro} implies
\formula[eq:poh]{
 \tag{$\heartsuit$}
 2 \lambda s \int_0^1 (\ph(x))^2 dx & = (\Gamma(1 + s))^2 \biggl(\lim_{x \to 1^-} (1 - x)^{-s} \ph(x)\biggr)^2 .
}
Suppose, contrary to our claim, that $\ph_1, \ph_2$ are orthogonal normalised solutions of~\eqref{eq:pro}, and let $\alpha_i = \lim\limits_{x \to 1^-} (1 - x)^{-s} \ph_i(x)$. Then~\eqref{eq:poh} for $\ph = \alpha_2 \ph_1 - \alpha_1 \ph_2$ reads
\formula{
 2 \lambda s (\alpha_1^2 + \alpha_2^2) & = (\Gamma(1 + s))^2 (\alpha_2 \alpha_1 - \alpha_1 \alpha_2)^2 = 0 .
}
On the other hand, \eqref{eq:poh} for $\ph_1$ and $\ph_2$ implies that $\alpha_1, \alpha_2 \ne 0$, a contradiction.
\end{proof}

The above proof is clearly related to the one given in~\cite{fgmp}: it can be thought of as the same argument, with all inessential elements removed. See Section~4 in~\cite{dfw} for a similar development for radial solutions in higher dimensions.

%
%

%
%

\end{document}